\documentclass[11pt,twoside,a4paper]{amsart}
\usepackage[french]{babel}
\usepackage{amssymb,bbm}
\usepackage{amsmath}
\usepackage{amsthm}
\usepackage{hyperref}
\usepackage[dvips]{graphics}
\usepackage{graphicx}
\usepackage{latexsym}
\usepackage[normal]{subfigure}
\input{epsf.sty}
\include{psfig}
\usepackage[all]{xy}
\usepackage{eufrak}
\usepackage{mathrsfs}
\usepackage{lipsum}
\usepackage{mdframed}
\usepackage{perpage}
\usepackage{amsfonts}
\usepackage{url,color,srcltx}
\begin{document}
\def\K{\mathbb{K}}
\def\R{\mathbb{R}}
\def\C{\mathbb{C}}
\def\Z{\mathbb{Z}}
\def\Q{\mathbb{Q}}
\def\D{\mathbb{D}}
\def\N{\mathbb{N}}
\def\T{\mathbb{T}}
\def\P{\mathbb{P}}
\def\A{\mathscr{A}}
\def\CC{\mathscr{C}}
\renewcommand{\theequation}{\thesection.\arabic{equation}}
\newtheorem{theorem}{Th{\'e}or{\`e}me}[section]
\newtheorem{lemma}{Lemme}[section]
\newtheorem{corollary}{Corollaire}[section]
\newtheorem{prop}{Proposition}[section]
\newtheorem{definition}{D{\'e}finition}[section]
\newtheorem{remark}{Remarque}[section]
\newtheorem{example}{Exemple}[section]
\newtheorem{notation}{Notation}
\newtheorem{con}{Cons{\'e}quence}
\bibliographystyle{plain}

\title[R{\'e}solution du $\partial \bar{\partial}$ ~~]{\textbf {R{\'e}solution du $\partial \bar{\partial}$ pour les formes diff{\'e}rentielles ayant une valeur au bord au sens des courants d{\'e}finies sur un domaine L{\'e}vi-plat non born{\'e} de} $\mathbb{C}^{n}$ } 
\author[ S.\  Sambou   \& S.\  Sambou \& W. O.\ Ingoba ]
{Souhaibou Sambou  \& Salomon Sambou \& Winnie Ossete Ingoba}

\address{D{\'e}partement de Math{\'e}matiques\\UFR des Sciences et Technologies \\ Universit{\'e} Assane Seck de Ziguinchor, BP: 523 (S{\'e}n{\'e}gal)}

\email{sambousouhaibou@yahoo.fr  \& ssambou@univ-zig.sn \& wnnossete@gmail.com }


\maketitle
    \begin{abstract}
On r{\'e}sout le $\partial \bar{\partial}$ pour les formes diff{\'e}rentielles admettant une valeur au bord au sens des courants d{\'e}finies sur un domaine L{\'e}vi-plat non born{\'e} de $\mathbb{C}^n$ dont son compl{\'e}mentaire est aussi L{\'e}vi-plat et non born{\'e} .
\vskip 2mm
\noindent
{\normalsize A}{\tiny BSTRACT.}
We solve the $\partial \bar{\partial}$-problem for a form with distribution boundary value on a Levi flat unbounded domain of $\mathbb{C}^n$ with the complementary is also Levi flat and unbounded.
\vskip 2mm
\noindent
\vskip 1.3mm
\noindent
\textit{Classification math{\'e}matique 2010~:}  32F32.
\end{abstract}        
\section*{Introduction}
Soit $\Omega \subset \subset \mathbb{C}^n$ un domaine. Dans cet article, on s'int{\'e}resse {\`a} la r{\'e}solution du $\partial \bar{\partial}$ dans le cadre des formes diff{\'e}rentielles ayant une valeur au bord au sens des courants.\\
Ayant Obtenue des r{\'e}sultats dans les cas o{\`u}:
\begin{enumerate}
\item[-] $\Omega$ est un domaine {\'e}toil{\'e} strictement pseudoconvexe de $\mathbb{C}^n$ et dans son compl{\'e}mentaire dans $[6]$.
\item[-] $\Omega$ est un domaine contractile compl{\'e}tement strictement pseudoconvexe d'une vari{\'e}t{\'e} analytique complexe et dans son compl{\'e}mentaire dans $[7]$.
\end{enumerate}  
On se demande si c'est possible d'obtenir des r{\'e}sultats dans le cas o{\`u} $\Omega$ est un domaine non born{\'e} de $\mathbb{C}^n$ dont son compl{\'e}mentaire est non born{\'e} et v{\'e}rifiant $H^j(\Omega)= H^j(b\Omega) = 0$ pour $1 \leq j \leq n-1$.\\
Pour r{\'e}pondre {\`a} cette question, la d{\'e}marche classique est de r{\'e}soudre\\ l'{\'e}quation $d u = f$ o{\`u} $u$ et $f$ sont des formes diff{\'e}rentielles ayant une valeur au bord au sens des courants et ensuite r{\'e}soudre le $\partial$ et le $\bar{\partial}$ pour les d{\'e}compositions de la solution obtenue. Il est donc int{\'e}ressant d'avoir des conditions topologiques et g{\'e}om{\'e}triques sur $\Omega$ pour obtenir des solutions du $d$, $\partial$ et du $\bar{\partial}$ pour les formes diff{\'e}rentielles ayant une valeur au bord au sens des courants. Le r{\'e}sultat obtenu dans ce sens est le suivant:
\begin{theorem} \label{AB}
Soit $\Omega = \{ z= (z_1, \cdots, z_n) \in \mathbb{C}^n : Im(z_n) > 0 \}$ un domaine, alors pour toute $(p,q)$-forme diff{\'e}rentielle $f$ de classe $C^\infty$, $d$-ferm{\'e}e admettant une valeur au bord au sens des courants et d{\'e}finie sur $\Omega$, il existe une $(p-1, q-1)$-forme diff{\'e}rentielle $g$ de classe $C^\infty$ d{\'e}finie sur $\Omega$ ayant une valeur au bord au sens des courants telle que $\partial \bar{\partial} g = f$.
\end{theorem}
\section{Pr{\'e}liminaires et notations}
\begin{definition} \label{zhou}  \item
Soit $\Omega \subset \subset \mathbb{C}^n$ un domaine {\`a} bord lisse de classe $C^\infty$ de fonction d{\'e}finissante $\rho$. Posons $\Omega_\varepsilon = \{z \in \Omega / \rho(z)<-\varepsilon\}$ o{\`u} $b\Omega_\varepsilon$ d{\'e}signe le bord de $\Omega_\varepsilon$.\\
Soit $f$ une fonction de classe $C^\infty$ sur $\Omega$. On dit que $f$ admet une valeur au bord au sens des distributions, s'il existe une distribution $T$ d{\'e}finie sur le bord $b\Omega$ de $\Omega$ telle que pour toute fonction $\varphi \in C^\infty (b\Omega)$, on ait:\\
\[ \lim_{\varepsilon \rightarrow 0} \int_{b\Omega_\varepsilon} f \varphi_\varepsilon d\sigma = <T,\varphi>\]
o{\`u} $\varphi_\varepsilon = i_{\varepsilon}^* \tilde{\varphi}$  avec $\tilde{\varphi}$ une extension de $\varphi$ \`a $\Omega$ et $i_{\varepsilon} : b\Omega_\varepsilon  \rightarrow \mathbb{C}^n$ l'injection canonique; $d\sigma$ d{\'e}signe l'{\'e}l{\'e}ment de volume.\\
Une forme diff{\'e}rentielle de classe $C^\infty$ sur $\Omega$ admet une valeur au bord au sens des courants si ses coefficients ont une valeur au bord au sens des distributions.
\end{definition}
\begin{definition}
Soit $\Omega \subset \subset \mathbb{C}^n$ un domaine. Un courant $T$ d{\'e}fini sur $\Omega$ est dit prolongeable s'il existe un courant $\tilde{T}$ d{\'e}fini sur $\mathbb{C}^n$ tel que $\tilde{T}_{\vert \Omega}= T$.
\end{definition}
\begin{notation}
Soit $\Omega \subset \subset \mathbb{C}^n$ un domaine, on note $\check{H}^r(\Omega)$ le $r^{i\grave{e}me}$ groupe de cohomologie de De Rham des courants prolongeables d{\'e}finis sur $\Omega$, $H^r(\Omega)$ le $r^{i\grave{e}me}$ groupe de cohomologie de De Rham des formes diff{\'e}rentiables de classe $C^\infty$ d{\'e}finies sur $\Omega$ et $H^r(b\Omega)$ le $r^{i\grave{e}me}$ groupe de cohomologie de De Rham des formes diff{\'e}rentiables de classe $C^\infty$ d{\'e}finies sur $b\Omega$. Le $r^{i\grave{e}me}$ groupe de cohomologie de De Rham des formes diff{\'e}rentielles ayant une valeur au bord au sens des courants sur  $\Omega$ est not{\'e} $\tilde{H}^r(\Omega)$.
\end{notation}
\section{R{\'e}solution du $d$ pour les formes diff{\'e}rentielles ayant une valeur au bord au sens des courants}
On consid{\'e}re $X = \mathbb{R}^{n+1}$ et
\begin{center}
$\Omega = \{ x \in \mathbb{R}^{n+1} : x_{n+1} > 0\} \subset \mathbb{R}^{n+1}$
\end{center}
un domaine contractile et $b\Omega = \{ x \in \mathbb{R}^{n+1} : x_{n+1} = 0\}$, son compl{\'e}mentaire $\Omega^{c} = \mathbb{R}^{n+1} \setminus \bar{\Omega} =\{ x \in \mathbb{R}^{n+1} :  x_{n+1} < 0 \}$.\\
$\Omega$ est un domaine convexe non born{\'e} et son compl{\'e}mentaire $\Omega^{c}$ est aussi convexe et non born{\'e}. On a donc $H^j(\Omega) = 0$ et $H^j(b\Omega) = 0$ pour $1 \leq j \leq n-1$.
\begin{theorem} \label{A}
Soit $f$ une $r$-forme diff{\'e}rentielle de classe $C^\infty$, $d$-ferm{\'e}e admettant une valeur au bord au sens des courants et d{\'e}finie sur $\Omega$. Alors il existe une $(r-1)$-forme diff{\'e}rentielle $g$ de classe $C^\infty$ d{\'e}finie sur $\Omega$ ayant une valeur au bord au sens des courants telle que $d g = f$.
\end{theorem}
\begin{proof}
D'apr{\`e}s $[4]$, si $f$ est une forme diff{\'e}rentielle ayant une valeur au bord au sens des courants sur $\Omega$ alors $[f]$ est un courant prolongeable. D'apr{\`e}s $[1]$, $\check{H}^r(\Omega) = 0$, il existe un $(r-1)$-courant prolongeable $u$ d{\'e}fini sur $\Omega$ tel que $du = f$. Soit $S$ une extension sur $\mathbb{R}^{n+1}$ de $u$ {\`a} support dans $\bar{\Omega}$, consid{\'e}rons le courant $F$ d{\'e}fini par $F = dS$ qui est un prolongement de $f$ sur $\mathbb{R}^{n+1}$.  D'apr{\`e}s ($[3]$ page $40$) on a 
\begin{center}
$S = RS + AdS + dAS$. 
\end{center}
Or $dS = F$
$\Rightarrow$ $dS = d(RS + AdS) = F$ donc $(RS + AdS)_{\vert \Omega}$ est une autre solution de l'{\'e}quation $du = f$. Or $(RS)$ est une forme diff{\'e}rentielle de classe $C^\infty$ sur $\bar{\Omega}$ donc admet une valeur au bord au sens des courants. Puisque $A$ n'augmente pas le support singulier et $dS_{\vert \Omega}$ est de classe $C^\infty$ alors $AdS_{\vert \Omega}$ est aussi de classe $C^\infty$. Donc la solution $RS+AdS$ est de classe $C^\infty$ sur $\Omega$. Il reste \`a montrer que $AdS$ admet une valeur au bord au sens des courants sur $\Omega$. 
Comme $\bar{\Omega}$ n'est pas born{\'e}, on consid{\'e}re la boule $\bar{B}(o,r)$ de $\mathbb{R}^{n+1}$ avec $\bar{B}(o,r) \cap b\Omega \neq \emptyset$.\\
 Donc $dS_{\vert (\bar{B}(o,r) \cap \bar{\Omega})}$ est un courant prolongeable d'ordre fini et $AdS_{\vert (\bar{B}(o,r) \cap \Omega)}$ admet une valeur au bord au sens des courants comme dans $[6]$.\\ Ainsi $(RS + AdS)_{\bar{B}(o,r) \cap \Omega)}$ admet une valeur au bord au sens des courants et $d(RS + AdS)_{\vert (\bar{B}(o,r) \cap \Omega)} = f$.\\
Soit $(\bar{B}(o,n))_{n \in \mathbb{N}}$ une famille de boule de $\mathbb{R}^{n+1}$ avec $\forall$ $n \in \mathbb{N}$, $\bar{B}(o,n) \cap b\Omega \neq \emptyset$ et $b\Omega \subset \displaystyle{\cup_{n \in \mathbb{N}}} \bar{B}(o,n)$. On a sur chaque $\bar{B}(o,n) \cap \Omega$, $RS + AdS$ admet une valeur au bord $V_n$ au sens des courants sur $\bar{B}(o,n) \cap b\Omega$.\\
Sur $\bar{B}(o,n+1) \cap \Omega$, $RS + AdS$ admet une valeur au bord $V_{n+1}$ au sens des courants sur $\bar{B}(o,n+1) \cap b\Omega$. Donc $d(V_{n+1}- V_n) = 0$ sur $\bar{B}(o,n) \cap b\Omega$. On a
\begin{center}
$\bar{B}(o,n) = \{x \in \mathbb{R}^{n+1} /  x_1^2 + x_2^2+ \cdots + x_{n+1}^2 \leq 1 \}$\\
\vspace{0,2cm}
$\bar{B}(o,n) \cap b\Omega = \{x \in \mathbb{R}^{n} /  x_1^2 + x_2^2+ \cdots + x_{n}^2 \leq 1 \}$
\end{center}
donc $\forall$ $n \in \mathbb{N}$, $\bar{B}(o,n) \cap b\Omega$ est un convexe de $\mathbb{R}^n$, alors il existe un $(0,r-1)$-courant $h_n$ d{\'e}fini sur $\bar{B}(o,n) \cap b\Omega$ tel que $V_{n+1}- V_n =  dh_n$.\\
Soit $\chi \in C^\infty(b\Omega)$  qui est identiquement {\'e}gale {\`a} $1$ sur $\bar{B}(o,n-1) \cap b\Omega$ et {\`a} support compact sur $\bar{B}(o,n+1) \cap b\Omega$.\\
\begin{center}
$V_{n+1} - d(1-\chi)h_n =  V_n + d( \chi h_n)$  sur $\bar{B}(o,n) \cap b\Omega$.
\end{center}
Posons $T_{n+1} = V_{n+1} - d(1 - \chi)h_n$ et $T_n = V_n + d(\chi h_n)$.\\
Ainsi 
\begin{center}
\[T =\lim_{n \rightarrow +\infty} T_{n}\]
\end{center}
est une valeur au bord de $RS + AdS$ au sens des courants sur $b\Omega$.\\
Posons $g =(RS + AdS)_{\vert \Omega}$, alors $g$ est une $(0,r-1)$-forme diff{\'e}rentielle de classe $C^\infty$ d{\'e}finie sur $\Omega$ ayant une valeur au bord au sens des courants et $dg = f$.
\end{proof}
\section{ R{\'e}solution du $\partial \bar{\partial}$ pour les formes diff{\'e}rentielles ayant une valeur au bord au sens des courants}
Consid{\'e}rons le cas o{\`u}\\
\begin{center}
$\Omega = \{ z= (z_1, \cdots, z_n) \in \mathbb{C}^n : Im(z_n) > 0 \}$
\end{center}
On donne le r{\'e}sultat suivant de r{\'e}solution du $\bar{\partial}$ pour les formes diff{\'e}rentielles ayant une valeur au bord au sens des courants:
\begin{theorem} \label{B}
Soient $\Omega = \{ z= (z_1, \cdots, z_n) \in \mathbb{C}^n : Im(z_n) > 0 \}$ et $f$ une $(0,r)$-forme diff{\'e}rentielle de classe $C^\infty$, $\bar{\partial}$-ferm{\'e}e admettant une valeur au bord au sens des courants et d{\'e}finie sur $\Omega$ avec $1 \leq r \leq n$. Il existe une $(0, r-1)$-forme diff{\'e}rentielle $g$ de classe $C^\infty$ d{\'e}finie sur $\Omega$ ayant une valeur au bord au sens des courants telle que $\bar{\partial} g = f$.
\end{theorem}
\begin{proof}
D'apr{\`e}s $[4]$, $[f]$ est un courant prolongeable. D'apr{\`e}s $[1]$, on a $\check{H}^{0,r}(\Omega) = 0$ donc il existe un courant prolongeable $u$ d{\'e}fini sur $\Omega$ tel que $\bar{\partial}u = f$. Soit $S$ une extension {\`a} $\mathbb{C}^n$ de $u$ {\`a} support dans $\bar{\Omega}$, consid{\'e}rons le courant $F$ d{\'e}fini par $F = \bar{\partial}S$ qui est un prolongement de $f$ {\`a} $\mathbb{C}^n$. D'apr{\`e}s la formule du $\bar{\partial}$-homotopie de $[2]$, on a 
\begin{center}
$S = R_\varepsilon S + A_\varepsilon F + \bar{\partial}A_\varepsilon S$.
\end{center}
$\Rightarrow$ $\bar{\partial} S = \bar{\partial}(R_\varepsilon S + A_\varepsilon F) = F$.
Ainsi $(R_\varepsilon S + A_\varepsilon F)_{\vert \Omega}$ est une autre solution de l'{\'e}quation $\bar{\partial} u = f$. $(R_\varepsilon S)$ est une forme diff{\'e}rentielle de classe $C^\infty$ {\`a} support dans $\bar{\Omega}$ donc admet une valeur au bord au sens des courants. L'op{\'e}rateur $A_{\varepsilon}$ est modulo un terme lisse qui est l'op{\'e}rateur $K$ de Bochner-Martinelli et puisque $f$ est de classe $C^\infty$ alors  $A_{\varepsilon} \bar{\partial}S_{\vert \Omega}$ est de classe $C^\infty$.\\
Il reste \`a montrer que $A_\varepsilon \bar{\partial}S$ restreinte {\`a} $\Omega$ admet une valeur au bord au sens des courants sur $\Omega$. Comme $\bar{\Omega}$ n'est pas born{\'e}, on consid{\'e}re un compact $K \subset \mathbb{C}^n$ avec $\overset{\circ}{K} \neq \emptyset$ et $\overset{\circ}{K} \cap b\Omega \neq \emptyset$. Donc $\bar{\partial}S_{\vert (K \cap \bar{\Omega})}$ est un courant prolongeable d'ordre fini et $A_\varepsilon \bar{\partial}S_{\vert (\overset{\circ}{K} \cap \Omega)}$ admet une valeur au bord au sens des courants comme dans $[5]$. Ainsi $(R_\varepsilon S + A_\varepsilon \bar{\partial}S)_{\vert (\overset{\circ}{K} \cap \Omega)}$ admet une valeur au bord au sens des courants et $\bar{\partial}(R_\varepsilon S + A_\varepsilon \bar{\partial}S)_{\vert (\overset{\circ}{K} \cap \Omega)} = f$.\\
Soit $(K_n)_{n \in \mathbb{N}}$ une famille exhaustive de compacts de $\mathbb{C}^n$ avec $\forall$ $n \in \mathbb{N}$ $\overset{\circ}{K}_n \neq \emptyset$ , $\overset{\circ}{K}_n \cap b\Omega \neq \emptyset$ et $b\Omega \subset \displaystyle{\cup_{n \in \mathbb{N}}} K_n$. On a sur chaque $\overset{\circ}{K}_n \cap \Omega$, $R_\varepsilon S + A_{\varepsilon} \bar{\partial}S$ admet une valeur au bord $U_n$ au sens des courants sur $\overset{\circ}{K}_n \cap b\Omega$.\\
Sur $\overset{\circ}{K}_{n+1} \cap \Omega$, $R_\varepsilon S + A_{\varepsilon} \bar{\partial}S$ admet une valeur au bord $U_{n+1}$ au sens des courants sur $\overset{\circ}{K}_{n+1} \cap b\Omega$.\\
Donc $\bar{\partial}_b(U_{n+1}- U_n) = 0$ sur $\overset{\circ}{K}_n \cap b\Omega$. Puisque le bord est L{\'e}vi-plat, alors sur chaque $\overset{\circ}{K}_n \cap b\Omega$, il existe un $(0,r-1)$-courant $V_n$ tel que $U_{n-1}- U_n =  \bar{\partial}_b V_n$
Soit $\chi \in C^\infty(b\Omega)$  qui est identiquement {\'e}gale {\`a} $1$ sur $\overset{\circ}{K}_{n-1} \cap b\Omega$ et {\`a} support compact sur $\overset{\circ}{K}_{n+1} \cap b\Omega$.\\
\begin{center}
$U_{n+1} - \bar{\partial}_b(1-\chi)V_n =  U_n + \bar{\partial}_b( \chi V_n)$  sur $\overset{\circ}{K}_n \cap b\Omega$.
\end{center}
Posons $h_{n+1} = U_{n+1} - \bar{\partial}_b(1 - \chi)V_n$ et $h_n = U_n + \bar{\partial}_b (\chi V_n)$.\\
Ainsi
\begin{center}
\[h =\lim_{n \rightarrow +\infty} h_{n}\]
\end{center}
est une valeur au bord de $R_\varepsilon S + A_\varepsilon \bar{\partial}S$ au sens des courants sur $b\Omega$.\\
Posons $g =(R _\varepsilon S + A_\varepsilon \bar{\partial}S)_{\vert \Omega}$, alors $g$ est une $(0,r-1)$-forme diff{\'e}rentielle de classe $C^\infty$ d{\'e}finie sur $\Omega$ ayant une valeur au bord au sens des courants et $\bar{\partial} g = f$.
\end{proof}
\begin{example}
Exemple d'une famille exhaustive compact de $\mathbb{C}^n$ est\\ $(\bar{B}(o,n))_{n \in \mathbb{N}}$.
\end{example}
\begin{proof}{(D{\'e}monstration du th{\'e}or{\`e}me \ref{AB})} \\
Soit $f$ une $(p,q)$-forme diff{\'e}rentielle de classe $C^\infty$, $d$-ferm{\'e}e d{\'e}finie sur $\Omega$ ayant une valeur au bord au sens des courants. D'apr{\`e}s le th{\'e}or{\`e}me \ref{A}, $\tilde{H}^{p+q}(\Omega) = 0$ alors il existe une $(p+q-1)$-forme diff{\'e}rentielle $u$ de classe $C^\infty$ ayant une valeur au bord au sens des courants telle que $du = f$.\\
Sans perte de g{\'e}n{\'e}ralit{\'e}, $u$ se d{\'e}compose en une $(p-1,q)$-forme diff{\'e}rentielle $u_1$ de classe $C^\infty$ ayant une valeur au bord au sens des courants et en une $(p,q-1)$-forme diff{\'e}rentielle $u_2$ de classe $C^\infty$ ayant une valeur au bord au sens des courants. On a\\
\begin{center}
$du = d(u_1 + u_2)= du_1 + du_2 = f$.
\end{center}
Comme $d = \partial + \bar{\partial}$, on a pour des raisons de bidegr{\'e} $\partial u_2 = 0$ et $\bar{\partial} u_1 = 0$. D'apr{\`e}s le th{\'e}or{\`e}me \ref{B}, il existe une $(p-1,q-1)$-forme diff{\'e}rentielle $h_1$ de classe $C^\infty$ ayant une valeur au bord au sens des courants telle que $\bar{\partial} h_1 = u_1$ et une $(p-1,q-1)$-forme diff{\'e}rentielle $h_2$ de classe $C^\infty$ ayant une valeur au bord au sens des courants telle que $\partial h_2 = u_2$. On a donc $\partial u_1 + \bar{\partial} u_2 =f$\\
\begin{center}
$\Rightarrow$ $\partial \bar{\partial}h_1 + \bar{\partial}\partial h_2 =f$\\
$\partial \bar{\partial}(h_1 - h_2) = f$.
\end{center}
Posons $g = h_1 -h_2$, alors $g$ est une $(p-1,q-1)$-forme diff{\'e}rentielle de classe $C^\infty$, d{\'e}finie sur $\Omega$ ayant une valeur au bord au sens des courants et $\partial \bar{\partial} g = f$.
\end{proof}

\end{document}